\documentclass[a4paper,oneside,11pt]{amsart}

\usepackage{enumerate}
\usepackage{comment}
\usepackage{xcolor}
\usepackage{mathtools}
\usepackage{amssymb}
\usepackage{amsmath}
\usepackage{amsbsy}
\usepackage{MnSymbol}	
\usepackage{calligra}
\usepackage{hyperref}					
\hypersetup{breaklinks, hidelinks,  
	pdfauthor={Marcus Appleby, Steven Flammia, Gary McConnell, Jon Yard},
	pdftitle={Generating Ray Class Fields of Real Quadratic Fields via Complex Equiangular Lines}}

\usepackage{tikz,tikz-cd}
\usetikzlibrary{matrix,arrows,decorations.pathmorphing}

\usepackage[vmargin=3cm]{geometry}

\newcommand\ZZ{\mathbb{Z}} 
\newcommand\RR{\mathbb{R}} 
\newcommand\CC{\mathbb{C}} 
\newcommand\NN{\mathbb{N}} 
\newcommand\QQ{\mathbb{Q}}
\newcommand\UU{\mathbb{U}}
\newcommand\VV{\mathbb{V}}

\newcommand\qD{{\sqrt{D}}}
\newcommand{\qr}[1]{{\QQ({#1})}}
\newcommand\zk{\ZZ_K}
\newcommand\stabo{\boldsymbol{S}}
\newcommand\nf{\small{\hbox{$\sqrt{d\!+\!1}$}}}
\newcommand\tinynf{\scriptstyle{\sqrt{d\!+\!1}}}
\newcommand\nft{\small{\hbox{$\sqrt{d\!-\!3}$}}}
\newcommand\invnf{\tfrac{1}{\sqrt{\!d\!+\!1}}}

\newcommand{\SIC}{SIC}
\newcommand{\SICs}{SICs}
\newcommand{\disp}[1]{{\small${#1}$\normalsize}}

\newcommand\g{{\gamma}}
\newcommand\Xd{X_d}
\newcommand\Zd{Z_d}
\newcommand\uf{u_f}
\newcommand\hild{\CC^d}
\newcommand\zdp{\ZZ/d'\ZZ}
\newcommand\spzd{{\mathrm{Sp}_2}(\zdp)}
\newcommand\espzd{{\mathrm{ESp}_2}(\zdp)}
\newcommand\glzd{{\mathrm{GL}_2}(\zdp)}
\newcommand\CS{\mathbf{C}\!\left(\stabo\right)}
\newcommand\MS{\mathbf{M}\!\left(\stabo\right)}
\newcommand\sD{{\sigma_D}}
\newcommand\lsD{{\hat{\sigma}}_D}
\newcommand\gsk{{ \Gal{\Sd}{K}}}
\newcommand\cK{\mathcal{C}_K}
\newcommand\mm{{\mathfrak{m}}}
\newcommand\mdp{{\mathfrak{d'}}}
\newcommand\ff{\mathfrak{f}}

\newcommand{\cheb}[2]{T_{#1}\left({#2}\right)}
\newcommand{\chsh}[2]{T^*_{#1}\left({#2}\right)}
\newcommand{\chebh}[1]{T_{#1}}
\newcommand{\chshh}[1]{T^*_{#1}}
\newcommand{\z}[1]{\zeta_{#1}}
\newcommand{\Rd}[1]{\mathfrak{R}_{#1}}
\newcommand{\Sd}{\mathcal{F}_1}
\newcommand{\zSd}{\ZZ_{\Sd}}
\newcommand{\Sdr}{\mathcal{F}}
\newcommand{\ipl}[1]{{\infty_{#1}}}
\newcommand{\Gal}[2]{{\mathrm{Gal}({#1/#2})}}
\newcommand{\normN}[3]{\mathrm{N}_{#1/#2}\left(#3\right)}
\newcommand{\modx}[2]{\equiv{#1}\!\!\!\mod^{\!\!\!\times}{#2}}
\newcommand{\ord}[2]{\mathrm{ord}_{#1}({#2})}

\newcommand\bfv{v} 
\newcommand\bfw{w} 
\newcommand\bfj{j} 

\newcommand\bP{\mathbb{P}}
\newcommand\SL{\mathrm{SL}}
\newcommand\GL{\mathrm{GL}}
\newcommand\Sp{\mathrm{Sp}}
\newcommand\ESp{\mathrm{ESp}}
\newcommand\U{\mathrm{U}}
\newcommand\EU{\mathrm{EU}}
\newcommand\PU{\mathrm{PU}}
\DeclareMathOperator{\Aut}{Aut}
\DeclareMathOperator{\EAut}{EAut}
\DeclareMathOperator{\Tr}{Tr}
\newcommand\mapsfrom{\leftmapsto} 
\newcommand\ph{\varphi}
\newcommand{\vev}[1]{\langle #1 \rangle}
\newcommand{\smat}[1]{{\bigl(\begin{smallmatrix}#1\end{smallmatrix}\bigr)}}

\newtheorem{theorem}{Theorem}
\newtheorem{proposition}[theorem]{Proposition}
\newtheorem{lemma}[theorem]{Lemma}
\newtheorem{corollary}[theorem]{Corollary}
\newtheorem{conjecture}[theorem]{Conjecture}
\newtheorem{remark}[theorem]{Remark}
\newtheorem{definition}[theorem]{Definition}

\newcommand{\blue}[1]{{\color{blue}{#1}}} 
\renewcommand{\blue}[1]{{#1}} 

\DeclareMathOperator{\rank}{rank}

\begin{document}

\baselineskip=17pt

\title[Generating Ray Class Fields of Real Quadratic Fields $\ldots$]{Generating Ray Class Fields of Real Quadratic Fields\\ via Complex Equiangular Lines}
\author[M. Appleby]{Marcus Appleby}
\address{Centre for Engineered Quantum Systems,  School of Physics, University of Sydney}
\email{marcus.appleby@sydney.edu.au}
\author[S. Flammia]{Steven Flammia}
\address{Centre for Engineered Quantum Systems,  School of Physics, University of Sydney}
\email{steven.flammia@sydney.edu.au}
\author[G. McConnell]{Gary McConnell}
\address{Controlled Quantum Dynamics Theory Group, Imperial College, London}
\email{g.mcconnell@imperial.ac.uk}
\author[J. Yard]{Jon Yard}
\address{Institute for Quantum Computing, Department of Combinatorics \& Optimization, University of Waterloo and Perimeter Institute for Theoretical Physics, Waterloo, ON}
\email{jyard@uwaterloo.ca}

\date{}

\begin{abstract}
For certain real quadratic fields~$K$ with sufficiently small discriminant we produce explicit unit generators for specific ray class fields of~$K$ using a numerical method that arose in the study of complete sets of equiangular lines in~$\hild$ (known in quantum information as symmetric informationally complete measurements or \SICs). 
The construction in low dimensions suggests a general recipe for producing unit generators in infinite towers of ray class fields above arbitrary real quadratic~$K$, and we summarise this in a conjecture.   There are indications~\cite{kopp,kopp2} that the logarithms of these canonical units are related to the values of~$L$-functions associated to the extensions, following the programme laid out in the Stark Conjectures.
\end{abstract}



\maketitle

\section{Introduction}\label{sec:intro}
Abelian class field theory enables us to reduce many questions about the abelian extensions of a number field~$K$ to the arithmetic of its ring of integers~$\ZZ_K$ by defining \emph{ray class fields} attached to its ideals. 
\blue{One of the great mysteries of this subject is that even though we can prove that these fields exist~\cite{Takagi}, we nevertheless cannot in general produce explicit generators for them except in the two special cases when $K$ is either~$\QQ$ or an imaginary quadratic field~\cite{cohstev}.  
Indeed the search for such generators goes back to Hilbert's 12th problem. 
In the present paper, we shall discuss a totally new approach for finding explicit generators when $K$ is a real quadratic field. } 

Let $D\geq 2$ be a square-free integer and let~$K$ be the real quadratic field~$\QQ(\sqrt D)$. 
Let~$d\geq4$ be an integer\footnote{The construction we describe also generates ray-class fields over $\QQ(\sqrt{D})$ when $d=2$ or $3$.  We confine ourselves to the case $d\geq 4$ because it is only then that the fields  are extensions of a real quadratic field.}~\ such that $D$ is the square-free part of~$(d-1)^2 - 4 = (d+1)(d-3)$. 
By Lemma~\ref{pellemma} below, there are infinitely many such~$d$ for any given~$D$, and 
conversely, every~$d\geq4$ corresponds to a non-trivial value of~$D$. 
Let~$d'=2d$ if~$d$ is even and~$d'=d$ if~$d$ is odd. 
This paper gives evidence that certain configurations of complex lines in $\CC^d$ yield natural generators for the ray class field $\Rd{}$ of $K$ modulo $d'$ with ramification allowed at both infinite places $\ipl{1}$ and $\ipl{2}$ of $K$, where we abbreviate the ideal~$d'\zk$ to~$d'$.
\blue{For convenience we give a brief introduction to ray class fields in section~\ref{RayClassSection}.} 

Equipping $\CC^d$ with the usual hermitian inner product $(v, w) = v^\dagger w$, we consider symmetric configurations of complex lines~\cite{zauner}, pairwise equiangular as measured by the hermitian angle 
\[\cos^{-1}\left(\frac{|(\bfv,\bfw)|}{\lVert\bfv\rVert \lVert\bfw\rVert}\right) = \cos^{-1}\left(\sqrt{\Tr \Pi_v \Pi_w}\right)\]
 between lines $\CC \bfv$, $\CC \bfw$,  where $\Pi_v = \frac{vv^\dagger}{(v,v)}$ is the rank-1 hermitian projector onto $\CC v$.
It is easily shown \cite{dgs} that $\CC^d$ can contain at most $d^2$ equiangular lines.  Following established terminology in quantum information theory \cite{renes}, we say that a complete set of $d^2$ equiangular lines is \emph{Symmetric and Informationally Complete} (\emph{SIC}), or a \emph{SIC}.  The common angle of a SIC is necessarily $\cos^{-1} \bigl(\frac{1}{\sqrt{d+1}}\bigr)$ \cite{dgs}.  

\blue{
At the time of writing, explicit constructions of SICs have been published in dimensions $d = 1$--$21$, $23$, $24$, $28$, $30$, $35$, $39$, $48$, $124$, $195$, $323$ while 
published numerical evidence suggests that they exist for all $d\le 151$ as well as $d=168$, $172$, $199$, $228$, $259$, and $844$. 
SICs are conjectured to exist in every finite dimension, forming a crucial part of what is known as Zauner's conjecture~\cite{zauner,renes}.}
With one exception~\cite{hoggar}, every known SIC is unitarily equivalent to an orbit of the group $(\ZZ/d\ZZ)^2$ under the projective unitary representation taking $j$ to the image of $X_d^{j_1} Z_d^{j_2}$ in the projective unitary group $\PU(d)$, where 
\begin{equation}\label{xandz}
\Xd = \left( 
\begin{array}{cccccc}
0 & 0 & 0 & \ldots & 0 & 1 \\
1 & 0 & 0 & \ldots & 0 & 0 \\
0 & 1 & 0 & \ldots & 0 & 0 \\
0 & 0 & 1 & \vspace{-2.5pt}\ldots & 0 & 0 \\
\vdots & \vdots & \vdots & \ddots & \ddots & \vdots \\ 
0 & 0 & 0 & \ldots & 1 & 0
\end{array}
\right) ,\quad 
\Zd = \left( 
\begin{array}{cccccc}
1 & 0 & 0 & 0 & \ldots & 0 \\
0 & \z{d} & 0 & 0 & \ldots & 0 \\
0 & 0 & \z{d}^2 & 0 & \ldots & 0 \\
0 & 0 & 0 & \z{d}^3 & \ldots & 0 \\
\vdots & \vdots & \vdots & \vdots & \ddots & \vdots \\ 
0 & 0 & 0 & 0 & \ldots & \z{d}^{d-1}\hspace{-5pt}
\end{array}
\right)
\end{equation}
and where $\zeta_d = e^{2\pi i /d}$ is a primitive $d$th root of unity. 
All  SICs considered in this paper will be orbits of $(\ZZ/d\ZZ)^2$ under the action of this representation.  Following the terminology usual in the physics literature we say  they are \emph{Weyl-Heisenberg covariant}. 
If $\bfv \in \CC^d$ is such that $\CC \bfv$ is contained in some SIC, we call $\bfv$ and its corresponding projector $\Pi_\bfv$ \emph{fiducial}.

Given a nonzero $\bfv\in \CC^d$, let $\QQ([\bfv]) = \QQ(v_i/v_j : v_j \neq 0)$ be the minimal field of definition of the corresponding point $[\bfv] \in \bP^{d-1}$ in projective space. 
For a given SIC $E$, let $\QQ(E)$ be the minimal field of definition of the corresponding subset of $\bP^{d-1}$, i.e.\ the smallest subfield of $\CC$ containing the fields $\QQ([\bfv])$ for every $\CC \bfv \in E$.  \blue{We refer to $\QQ(E)$ as the SIC-field.}
Then we have the following result and subsequent natural conjecture:

\begin{proposition}\label{rcfobs}
Let $d$ be any one of the following dimensions: 
\begin{equation}\label{dimz}
	4, 5, 6, 7, 8, 9, 10, 11, 12, 13, 14, 15, 16, 17, 18, 19, 20, 21, 24, 28, 30, 35, 39, 48.
\end{equation}
\blue{Then $\CC^d$ contains at least one SIC for which the corresponding SIC field is equal to  $ \mathfrak{R}$, the ray class field of~$K = \QQ(\sqrt D)$ with conductor~$d'\ipl{1}\ipl{2}$. 
Moreover, for every other known SIC in dimension $d$ the  SIC field  is a finite Galois extension of $\QQ$ containing $\mathfrak{R}$.}
\end{proposition}
\blue{
\begin{proof}
The proof involves lengthy computations using \emph{Magma}~\cite{magma}. 
The main steps are indicated in Section~\ref{very}.
\end{proof}
}
\begin{conjecture}\label{rcfconjecture}
The statement of Proposition~\ref{rcfobs} holds for every dimension $d\ge 4$.
\end{conjecture}

\noindent These statements are a substantial strengthening of results and conjectures in~\cite{AYZ}.
\blue{As a further strengthening, at the end of Section~\ref{sec:Weil} we give a putatively complete classification of the SICs in the dimensions of Eq.~(\ref{dimz}) for which the SIC field equals $\mathfrak{R}$. 
The completeness of the classification hinges on the numerical work of Scott and Grassl~\cite{scottgrassl}, who have with high probability identified all the (Weyl-Heisenberg covariant) SICs in these dimensions.
}

After introducing the necessary geometric and number-theoretical apparatus in the next two sections, we informally introduce ray class fields and establish some properties of the relevent fields in \S\ref{RayClassSection}. 
In \S\ref{GandA} and \S\ref{canon} we relate invariants of the geometric objects to arithmetic quantities in the ray class fields and to associated canonical units.
\blue{Section~\ref{very} outlines the proofs of Proposition~\ref{rcfobs} above, and Propositions~\ref{prpconstfield},~\ref{units},~\ref{prp:ugp} below.}

\blue{\subsection{Subsequent developments and marginalia}

Since the first version of this paper appeared, explicit solutions have been published~\cite{AB,grasslscott,kopp2} in the further dimensions $23$, $124$, $195$ and $323$. 
These new solutions were in whole or in part directly inspired by the first version of this paper. 
Our original analysis has not yet been extended to include these, hence the list of dimensions in Eq.~(\ref{dimz}) is not exhaustive with respect to the known solutions.}

\blue{To see the actual solutions themselves, the reader should first consult Scott and Grassl~\cite{scottgrassl}. 
In addition to many solutions found by them, they give a comprehensive listing of solutions, both explicit and numerical, published by previous authors. 
For solutions discovered subsequently see Refs.~\cite{ABBGGL,ABBEGL,AB,fuchs,grasslscott,kopp2,scott}.  
Many solutions are available in conveniently downloadable form on the websites~\cite{exactSICsDB,zaunerDB}.}

\blue{Finally, it is worth elaborating our previous claim that the only known SICs (with the exception of the Hoggar lines~\cite{hoggar}) are Weyl-Heisenberg covariant SICs. 
Historically, at first it appeared that there exist other examples of non-Weyl-Heisenberg covariant SICs~\cite{grassl,renes}. 
However, it was then discovered~\cite{huangjun2} that these other ostensible exceptions are covariant with respect to more than one group, one of them being the Weyl-Heisenberg group. 
It is, in general, an open question whether there exist examples of non-Weyl-Heisenberg covariant SICs yet to be discovered. 
However, it can be proved~\cite{huangjun,huangjun2} that if $d$ is prime, and if a SIC has a group covariance property at all, then it must be Weyl-Heisenberg covariant.}

\section{SIC orbits under the extended Weil representation \label{sec:Weil}}
A subgroup $H \subset \U(d)$ sitting in a central extension 
\begin{equation} \label{eqn:Heis}
1 \to \U(1) \to H  \to (\ZZ/d\ZZ)^2 \to 0
\end{equation}
canonically determines a projective unitary representation $(\ZZ/d\ZZ)^2 \to \PU(d)$ via $j \mapsto [\Delta_j]$, where $\Delta_j \mapsfrom j$ is any choice of section, and where $[U] \in \PU(d)$ denotes the image of $U\in \U(d)$ in the projective unitary group. 
The $\Delta_j$ determine a  2-cocycle $(j,k)\mapsto c_{j,k} \in \U(1)$ via $\Delta_j \Delta_k = c_{j,k}\Delta_{j+k}$, by definition satisfying
 $c_{k,\ell}c_{j,k+\ell} = c_{j + k,\ell}c_{j,k}$.  In turn, this cocycle determines a skew-multiplicative pairing $(j,k) \mapsto e_{j,k} = c_{j,k} / c_{k,j}$, which satisfies
$\Delta_j \Delta_k \Delta_j^{-1} \Delta_k^{-1} = e_{j,k}I$, where $e_{j,k} \in \vev{\zeta_d}$ for every $j,k$.  It is independent of the choice of section, hence an invariant of the extension. 
The extension (\ref{eqn:Heis}) is a \emph{Heisenberg group} if $e$ is nondegenerate (i.e.\ induces an isomorphism between $(\ZZ/d\ZZ)^2$ and its dual, whence $e((\ZZ/d\ZZ)^2, (\ZZ/d\ZZ)^2) = \vev{\zeta_d}$.  

From now on let $H = \{\phi X_d^{j_1} Z_d^{j_2} : \phi \in \U(1), j \in (\ZZ/d\ZZ)^2\}$. 
Then $H \subset \U(d)$ is a Heisenberg group and the \emph{Stone-von Neumann theorem} \cite{Stone, vonNeumann} asserts that every unitary representation $\rho : H \to \U(d)$ acting as the identity on scalars has the form $\rho(h)  = U_\rho h U_\rho^{-1}$, for some $U_\rho \in \U(d)$ that is uniquely determined up to scalars. 
Thus follows the existence of the \emph{Weil representation}~\cite{Weil}, a canonical projective unitary representation of the group $\Aut^0(H)$ of automorphisms of $H$ fixing its center, taking $t \in \Aut^0(H)$ to $[U_t]$ for any $U_t$ satisfying $t(h) = U_t^{\vphantom{\dagger}} h U_t^\dagger$ for every $h \in H$. 
In quantum computing, $\Aut^0(H)$ is known as the \emph{projective generalized Clifford group} and the normalizer of $H$ in $\U(d)$ is known as the \emph{generalized Clifford group}. 
This group sits in an extension~\cite{prasad}
\begin{equation}\label{eqn:Aut0H}
0 \to (\ZZ/d\ZZ)^2 \to \Aut^0(H) \to \Sp_2(\ZZ/d\ZZ) \to 1
\end{equation}
of the subgroup $\Sp_2(\ZZ/d\ZZ) = \SL_2(\ZZ/d\ZZ)$ of $\GL_2(\ZZ/d\ZZ)$ fixing $e$. 
The action of $t\in \Aut^0(H)$ on a section $\Delta_j \mapsfrom j$ is given by 
\cite{prasad}
\begin{equation} \label{eqn:HAut0}
t(\Delta_j) = \ph_t(j) \Delta_{F_t j},
\end{equation} 
where $\ph_t \colon (\ZZ/d\ZZ)^2 \to \U(1)$ is a function and $t \mapsto F_t$ is a surjective homomorphism.  The injection in (\ref{eqn:Aut0H}) is defined by $j \mapsto (h \mapsto \Delta_j h \Delta_{-j})$ and gives an action of $(\ZZ/d\ZZ)^2$ that depends only on the pairing:
\[j(\Delta_k) = \Delta_j \Delta_k \Delta_j^{-1} = e_{j,k}\Delta_k.\]

For the section $j \mapsto \Delta_j  = \phi_j X_d^{j_1} Z_d^{j_2}\in H$ determined by phases $\phi_j \in U(1)$, 
the cocycle is given by $c_{j,k} = \phi_j\phi_k\phi_{j+k}^{-1} \zeta_d^{j_2k_1}$, while the pairing is independent of the phases: $e_{j,k} = c_{j,k}/c_{k,j} = \zeta_d^{k_1 j_2 - k_2 j_1}$.  
The specific choice $\phi_j = (-\zeta_{2d})^{j_1 j_2}$ of phases gives the section
\begin{equation}
\Delta_{\bfj} = (-\z{2d})^{j_1j_2}\Xd^{j_1}\Zd^{j_2}
\label{DisOpDef}
\end{equation}
used in \cite{AYZ, scottgrassl} and in \S\ref{GandA} and \S\ref{canon}, which satisfies
 $\Delta_j^{-1} = \Delta_{-j}$ and $c_{j,k} = (-\zeta_{2d})^{j_2 k_1 - j_1 k_2} \in \vev{\zeta_{d'}}$.
When $d$ is odd, then  $-\zeta_{2d} = \zeta_d^{2^{-1} \textrm{ mod } d}$, in which case the following hold \cite{prasad}: $e$ and $c$ have the same isometry group, the phase in (\ref{eqn:HAut0}) is a  homomorphism $\ph_t \colon (\ZZ/d\ZZ)^2 \to \vev{\zeta_d}$ and the extension (\ref{eqn:Aut0H}) splits as a semidirect product $\Aut^0(H) \simeq (\ZZ/d\ZZ)^2\rtimes \Sp_2(\ZZ/d\ZZ)$ via the injective homomorphism taking $F\in \Sp_2(\ZZ/d\ZZ)$ to the automorphism of $H$ defined by $F(\Delta_j) = \Delta_{Fj}$. 
This splitting identifies each automorphism $t\in \Aut^0(H)$ with a pair $t=(j_t,F_t)$.  We henceforth identify $(\ZZ/d\ZZ)^2$ and $\Sp_2(\ZZ/d\ZZ)$ with their images $j \mapsto (j,1)$ and $(0,F) \mapsfrom F$ in $\Aut^0(H)$ under this splitting. The extension (\ref{eqn:Aut0H}) does not split when $4\mid d$, but 
for general $d$, there is a surjective homomorphism $(\ZZ/d\ZZ)^2 \rtimes \Sp_2(\ZZ/d'\ZZ) \to \Aut^0(H)$ that is an isomorphism when $d$ is odd, but has kernel $\simeq (\ZZ/2\ZZ)^3 \simeq \ker(\SL_2(\ZZ/d'\ZZ) \to \SL_2(\ZZ/d\ZZ))$ when $d$ is even~\cite{appleby05}. 

Let $\EAut^0(H)$ be the group of automorphisms of $H$ acting either as the identity or by complex conjugation on the center. 
It may be viewed as the projectivization of the \emph{extended Clifford group}, defined in \cite{appleby05} as the normalizer of $H$ in the group $\EU(d)$ of all unitary and anti-unitary operators in dimension $d$~\cite[Chap.~26]{wigner}. 
It is generated by $\Aut^0(H)$ together with complex conjugation in the standard basis.  
It lies in an extension 
\[0 \to (\ZZ/d\ZZ)^2 \to \EAut^0(H) \to \ESp_2(\ZZ/d\ZZ) \to 1\]
of the determinant $\pm 1$ subgroup $\ESp_2(\ZZ/d\ZZ) \subset \GL_2(\ZZ/d\ZZ)$, 
with complex conjugation mapping to $\smat{1 & 0 \\ 0 & -1}$.
There is similarly a surjective homomorphism $(\ZZ/d\ZZ)^2 \rtimes \ESp_2(\ZZ/d'\ZZ) \to \EAut^0(H)$.
\blue{The SICs we consider are orbits under the action of $(\ZZ/d\ZZ)^2$.  It follows that in a given dimension they  
decompose into a set of  orbits of $\ESp_2(\ZZ/d'\ZZ)$.
Aside}\footnote{In $d=3$ there is an uncountably infinite pencil of inequivalent \SICs{} with (necessarily) transcendental fields of definition apart from special points~\cite[\S10]{AYZ},\cite{huangjun}. For an in-depth look at the case $d=3$, see also~\cite{hughston}.} from dimension $3$, there are only finitely many such orbits in every known case. 

\blue{
\begin{proposition}\label{prpconstfield}:
The SIC field is constant on each orbit of $\EAut^0(H)$.  In particular, if Proposition~\ref{rcfobs} holds for one fiducial on a given orbit of $\EAut^0(H)$ then it holds for all.  
\end{proposition}
\begin{proof}
See Section~\ref{very}.
\end{proof}
}
We define the \emph{stabilizer group} of a fiducial as the subgroup of $(\ZZ/d\ZZ)^2\rtimes \ESp_2(\ZZ/d'\ZZ)$ leaving it invariant. 
Every known orbit contains a fiducial such that
\begin{itemize}
\item The stabilizer group $\boldsymbol{S}_0$ is contained in $\ESp_2(\ZZ/d'\ZZ)$.
\item $\boldsymbol{S}_0$ contains an order-3 element of $\spzd$ with trace $\equiv -1\bmod d$.
\end{itemize}
We refer to such a fiducial as \emph{centred} (in \cite{AYZ} such fiducials were called ``simple''). 
As observed in \cite{scottgrassl}, the order-3 element can always be chosen to be conjugate to either $F_z = \smat{0 & d-1 \\ d+1 & d-1}$ or 
$F_a = \smat{1 & d+3 \\ 4d/3 - 1 & d-2}$, the latter possibility only occurring when $d\equiv 3 \bmod 9$. 
We note that $\boldsymbol{S}_0$ is in fact always cyclic in the known instances for $d\ge 4$.

In the enumeration scheme of~\cite{scottgrassl} the orbits for $d\ne 3$ are specified by the dimension followed by a letter, and for future reference we also list here the type of stabilizer ($F_z$ or $F_a$). 
For dimensions in the list (\ref{dimz}), the orbits that are known to generate $\mathfrak{R}$ over the rationals are as follows:
\begin{align}
[\text{$F_z$-type}]\ & 4a, 5a, 6a, 7b, 8b, 9ab, 10a, 11c, 13ab, 14ab, 15d, 16ab, 17c, 18ab, 19e, 
\label{rcfdimz}
\\ 
&\hspace{8 cm} 20ab, 24c, 28c, 35j, \nonumber\\
\label{rcfdima}
[\text{$F_a$-type}]\ & 12b, 21e, 30d, 39ij, 48g.
\end{align}
Here, two or more letters indicates several orbits---for instance $13ab$ is shorthand for $13a$, $13b$ and $39a\text{-}f$ means $39a, \ldots, 39f$. 
\blue{Scott and Grassl~\cite{scottgrassl} have, with high probability, identified every orbit for $d\le 48$ (although in many cases the solutions are only known numerically). For the dimensions in list~\eqref{dimz} the  orbits they identify additional to those in lists~\eqref{rcfdimz} and~\eqref{rcfdima} are
\begin{align}
[\text{$F_z$-type}]\ & 7a, 8a, 11ab, 12a, 15abc, 17ab,  19abcd, 21abcd, 24ab, 28ab, 30abc, 
\label{remorbsz}
\\ 
&\hspace{7 cm} 35a\text{-}i, 39a\text{-}f, 48abcdf, \nonumber
\\ 
 [\text{$F_a$-type}]\ & 39gh, 48e.
\label{remorbsa} 
\end{align}
Explicit solutions have been calculated for all of these orbits.  In every case it is found that the field generated contains, and is strictly larger than $\mathfrak{R}$.  For a detailed analysis of these fields, and those generated by the SICs in lists~\eqref{rcfdimz} and~\eqref{rcfdima}, see refs.~\cite{AYZ,MAIB} .

Finally let us note that for lists~\eqref{rcfdimz} and~\eqref{rcfdima} the number of orbits in each dimension is the same as the class number of $K$ for that dimension. 
This is likely not an accident, as we discuss in \S\ref{GandA}. 
}

\section{Pell's equation and towers of ray class fields over~\texorpdfstring{$K$}{K}}\label{pell}
We now prove the assertions made in the introduction about infinite sequences of dimensions $d_1,d_2,\ldots$ all of which give rise to the same value of~$D$. 
\blue{As we shall now show, the dimensions~$d_k$ are indexed precisely by the solutions of the generalised Pell's equation 
$$X^2-DY^2 = 4,$$
which form a rank-1~$\ZZ$-module by the theory of binary quadratic forms~\cite{harv}. }
Conjecturally these give us `towers' of ray class fields lying above each quadratic field~$K$.

Fix~$D$,~$K=\qr{\qD}$ as above and write~$\ipl{1}$ for the place associated to the embedding of~$K$ into~$\RR$ under which~$\qD$ is sent to the positive square root of~$D$. 
Let~$\uf\in\zk^\times$ be the fundamental unit of~$K$ which is~$>1$ under~$\ipl{1}$. 
If the norm of~$\uf$ is~$-1$ then set~$u_D=\uf^2$; otherwise set~$u_D=\uf$. 
In other words,~$u_D$ is the first power of~$\uf$ of norm~1.

\begin{lemma}\label{pellemma}
The following are equivalent.
\begin{itemize}
	\item[(i)] $d$ \it is a positive integer such that the square-free part of~$(d-1)^2 - 4$ is equal to~$D$
	\item[(ii)] $\frac{d-1}{2}$ \it is the rational part of~$u_D^r$ for some~$r\in\NN$.
\end{itemize}
\end{lemma}

\begin{proof}
Statement~(i) is equivalent to saying there exists~$y\in\ZZ$ such that~$\left(\tfrac{d-1}{2}\right)^2-D(\tfrac{y}{2})^2=1$, 
which in turn is equivalent to~$\frac{d-1}{2}+\frac{y}{2}\qD$ being an element of~$K$ of norm~1. 
But it is also a root of the monic integral polynomial~\hbox{$X^2-(d-1)X+1$}, hence a unit. 
Therefore~\cite[\S 11B]{harv} it is a power of~$u_D$, proving~(ii). 
The converse is just a restatement of the definitions.
\end{proof}

\begin{corollary}
For every square-free value of~$D\geq2$ there exist infinitely many values of~$d$ such that the square-free part of~$(d-1)^2 - 4 = (d+1)(d-3)$ is equal to~$D$.
\end{corollary}

\begin{definition}
All~$d$ yielding the same value of~$D$ are indexed by the appropriate power of~$u_D$:
\begin{equation}\label{dr}
d_r=1+u_D^r+u_D^{-r}.
\end{equation}
\end{definition}
\noindent In particular,~$d_1$ is the smallest dimension corresponding to~$D$. 
We shall generally omit reference to the underlying~$D$ but~\emph{the notation~$d_r$ always applies with reference to a specific fixed value of~$D$}. 
Note that the rational part of~$u_D^r$ is given by~$\tfrac{u_D^r+u_D^{-r}}{2}=\cosh{(r\log u_D)}$.
The Chebyshev polynomials of the first kind~$\chebh{n}$ tell us how to go from the rational part of~$u_D^r$ to that of~$u_D^{rs}$:
\begin{align}\label{gosh}
\cheb{s}{\tfrac{d_r-1}{2}} \ = \ \cheb{s}{\cosh{(r\log u_D)}} \ = \ \cosh{(sr\log u_D)} \ = \ \tfrac{d_{rs}-1}{2} \ = \ \cheb{r}{\tfrac{d_s-1}{2}}.
\end{align}

\noindent
Let us define a shifted version of the functions~$\chebh{n}$ by
\[
\chsh{n}{x} = 1+2\cheb{n}{\tfrac{x-1}{2}},
\]
and for convenience extend it to negative~$n$ by defining~$\chshh{-n} = \chshh{n}$. 
Equation~(\ref{gosh}) may be rephrased in terms of the new functions~$\chshh{n}$ and rearranged to read:
\begin{equation}\label{twirl}
\chsh{r}{\chsh{s}{d_1}} = \chsh{r}{d_s} = d_{rs} = d_{sr} = \chsh{s}{d_r} = \chsh{s}{\chsh{r}{d_1}}.
\end{equation}
It is also easy to show that~$d_1,d_2,d_3,\ldots$ is a strictly increasing sequence of positive integers. 

\begin{proposition}\label{towers}
Within each sequence $d_1,d_2,\dotsc$ of dimensions corresponding to a fixed~$D$, there are infinitely many distinct infinite subsequences $d_{k_1},d_{k_2},\dotsc$ with the property that~$d_{k_1}\mid d_{k_2}\mid \cdots $.
\end{proposition}

\begin{proof}
Fix~$D$. 
The defining recursion 
\[
\cheb{0}{x} = 1,\,\,\, \cheb{1}{x} = x, \,\,\, \cheb{n}{x} = 2x\cheb{n-1}{x} - \cheb{n-2}{x}
\]
for the Chebyshev polynomials yields the following recursion for the shifted version:
\begin{equation*}
\chsh{0}{x} = 3,\,\,\, \chsh{\pm 1}{x} = x, \,\,\,\chsh{n}{x} = x\chsh{n-1}{x} - x\chsh{n-2}{x} + \chsh{n-3}{x}.
\end{equation*}
If~$x\in\ZZ$ then all terms are in~$\ZZ$, so for the particular case of a positive integer~$d$,
\begin{equation}\label{redux}
\chsh{n}{d}\equiv\chsh{n-3}{d}\bmod d,
\end{equation}
and in fact for any positive integer~$r$ we have $d_{nr}\equiv d_{(n-3)r} \bmod d_r$.
We now have three cases according to the congruence class of~$n$ modulo~3: 

\begin{itemize}
\item \textbf{C0:} $n\equiv0\ \bmod 3$:\ \ $\chsh{0}{d_r}=3$, implying~$d_{nr}-3$ is a multiple of~$d_r$.

\item \textbf{C1:} $n\equiv1\ \bmod 3$:\ \ $\chsh{1}{d_r}=d_r$, implying~$d_{nr}$ is a multiple of~$d_r$.

\item \textbf{C2:} $n\equiv2\ \bmod 3$:\ \ $\chsh{2}{d_r}=d_r(d_r-2)$, implying~$d_{nr}$ is a multiple of~$d_r$.
\end{itemize}
So if $(i_j)_{j\geq1}$ is any increasing sequence of integers coprime to $3$ such that $i_1 \mid i_2 \mid i_3 \mid \cdots$, it follows that $d_{i_1} \mid d_{i_2}\mid d_{i_3} \mid \cdots$.\ \ 
There are infinitely many such dimension towers because $d_s > d_r > 3 \ \forall\ s>r$, so by~\textbf{C0},~$n> m\implies d_{3^m i_k}\nmid d_{3^n i_j}$ for any $j, k$.
\end{proof}

\begin{corollary}
If Conjecture~\ref{rcfconjecture} is true, then for every square-free $D\geq2$, there exist infinitely many infinite ray class field towers above each~$\qr{\qD}$ whose successive generators may be found by constructing SICs in the $\CC^{d_{k_i}}$. 
\end{corollary}

\section{Properties of the ray class fields}\label{RayClassSection}

\subsection{Definition of ray class fields}
\blue{
By the Kronecker-Weber theorem~\cite{cohstev}, any abelian extension of~$\QQ$ is a subfield of some cyclotomic extension~$\QQ(\zeta_m)$, where~$\zeta_m$ is a primitive~$m$-th root of unity for some minimal positive integer~$m$.  
By the fundamental theorem of Galois theory, therefore, its Galois group is isomorphic under the Artin map to a quotient of~$\left( \ZZ/m\ZZ \right)^\times$, the multiplicative group of the quotient ring~$\ZZ/m\ZZ$. 
The ideal~$m\ZZ$ is called the \emph{conductor} of the extension: and indeed~$\QQ(\zeta_m)$ is the \emph{narrow ray class field of~$\QQ$ modulo~$m\ZZ$}.  
The unique infinite prime of~$\QQ$ is allowed to~\emph{ramify} here and we obtain a totally complex field.  
If on the other hand we were to restrict ramification just to the finite places dividing~$m$ then the ray class field modulo the same conductor~$m\ZZ$ would be the maximal totally real subfield~$\QQ(\zeta_m+\zeta_m^{-1})$ of~$\QQ(\zeta_m)$.  

The classical ideal-theoretic version of Artin's global reciprocity law for finite abelian extensions of any finite extension~$K$ of~$\QQ$ then leads naturally to the more general notion of the ray class fields of~$K$ modulo any conductor. 
Indeed, let~$F$ be any finite abelian extension of~$K$. 
Define~$S$ as the set of finite and infinite primes of~$K$ which ramify in~$F$; let~$I_S$ be the free abelian group on the finite primes of~$K$ which do not lie in~$S$. 
Then the classical Artin map is the homomorphism~$\phi_{F/K}: I_S \to \Gal{F}{K}$ which sends an unramified finite prime~$v$ of~$K$ to the Frobenius automorphism of~$v$ in~$\Gal{F}{K}$. 
Artin's global reciprocity law asserts that~$\phi_{F/K}$ is surjective; and moreover that there exists an integral ideal~$\ff$ of~$K$ divisible precisely by the finite primes in~$S$, such that the kernel of~$\phi_{F/K}$ always contains the group~$P_\ff$ of principal ideals generated by elements~$\alpha$ of $K$ satisfying: 
\begin{itemize}

\item $
\ord{v}{\alpha -1} \geq \ord{v}{\ff}  
$ for all~$v$ dividing~$\ff$; and

\item $\alpha$ is totally positive at the real primes of~$K$ which ramify in~$F$.
\end{itemize}
The smallest such~$\ff$ is called the conductor of~$F/K$. 
The quotient group~$I_S/P_\ff$ is called the~\emph{ray class group of~$K$ modulo~$\ff$}, and thus Artin's law asserts that the ray class group modulo~$\ff$ is mapped onto~$\Gal{F}{K}$ by the Artin map. 
Moreover, for every such ideal~$\ff$, the existence theorem of global class field theory shows that there exists a finite abelian extension~$F_\ff/K$, unramified outside~$S$, such that the kernel of the Artin map for this extension is precisely~$P_\ff$.  
This field, unique up to isomorphism, is called the~\emph{ray class field of~$K$ modulo~$\ff$}. 
}

\blue{ 
\subsection{Properties of ray class fields in the context of SICs}  
Fix a square-free integer}~$D\geq2$ and an associated dimension~$d\geq4$ \blue{as explained in the previous section, defining as always our base field to be~$K = \QQ(\sqrt{D})$. }
Recall that~$d'=d$ if~$d$ is odd or~$2d$ if~$d$ is even, and that~$\z{d'}$ is an arbitrary primitive~$d'$-th root of unity. 
Write~$\mdp$ for the modulus~$d'\ipl{1}\ipl{2}$ and let~$\!\!\!\!\mod^{\!\!\!\times}\!$ denote multiplicative congruence. 
Let~$\sD$ be the unique non-trivial element of~$\Gal{K}{\QQ}$. 
Write~$\Rd{}$ for the ray class field of~$K$ modulo~$d'$ with ramification allowed at both infinite primes; similarly we write~$\Rd{0},\Rd{1}$ and~$\Rd{2}$ when ramification is allowed respectively at no infinite primes, just at~$\ipl{1}$, and just at~$\ipl{2}$. 
\blue{We have illustrated this notation and many of the conclusions of the next proposition in the field tower diagram in Figure~\ref{Figfieldtower}.  
}

\begin{proposition}\label{rcfbasics}
\begin{itemize}
\item[(i)] $\Rd{}$ contains the~$d'$-th roots of unity~\blue{$\vev{\zeta_{d'}}$}, and~\disp{\z{d'}+\z{d'}^{-1}\in\Rd{0}}. 

\item[(ii)] $\Rd{0}$ is a non-trivial extension of~$K$.  Hence so are~$\Rd{1},\Rd{2},\Rd{}$.  

\item[(iii)] $\Rd{0}$ and~$\Rd{}$ are Galois over~$\QQ$.  

\item[(iv)] $\Rd{1}$ and~$\Rd{2}$ have isomorphic Galois groups over~$K$. 

\item[(v)] $\Rd{}$ is the compositum~$\Rd{1}\Rd{2}$, and~$\Rd{0} = \Rd{1}\bigcap\Rd{2}$. 
\end{itemize}
\end{proposition}

\begin{proof}
$\Rd{}$ is the fixed field of the Artin symbols~\disp{\left((\alpha),\Rd{}/K\right)} attached to the integral principal ideals~$(\alpha)$ where~$\alpha\modx{1}{\mdp}$. 
The extension~\blue{\disp{K(\zeta_{d'})/K}} is unramified outside primes dividing~$\mdp$, so to show that~\blue{\disp{\zeta_{d'}\in\Rd{}}} it suffices to show that these Artin symbols fix \blue{it}. 
This action translates~\cite[X \S1]{lang} via the restriction map from~\blue{\disp{K(\zeta_{d'})/K}} down to~\blue{\disp{\QQ(\zeta_{d'})/\QQ}}, to raising~$\z{d'}$ to the power~\disp{\normN{K}{\QQ}{\alpha}}. 
But~$\mdp$ is fixed by~$\sD$ and so~(i) follows from:
\[
\alpha\modx{1}{\mdp} \implies \alpha^\sD\modx{1}{\mdp}\implies
\alpha\alpha^\sD\modx{1}{\mdp} \implies \normN{K}{\QQ}{\alpha}\equiv1\ (d'). 
\]
As a direct consequence,~\disp{\z{d'}+\z{d'}^{-1}\in\Rd{0}}. 
It follows that to prove~(ii) we  need only show that~\disp{\z{d'}+\z{d'}^{-1}\not\in \qr{\qD}}, which follows from the definitions of~$D,d'$ and basic cyclotomic theory. 

Now lift~$\sD$ to some~\disp{\lsD\in\Gal{\overline{\QQ}}{\QQ}}.
In terms of moduli,~$\lsD$ fixes~$d'$ and~$\mdp$ and interchanges~$d'\ipl{1}$ and~$d'\ipl{2}$. 
The field~\disp{\lsD\left(\Rd{}\right)} sits inside a normal closure of~\disp{\Rd{}/\QQ}: it is an abelian extension of~$K$ with conductor~$\mdp$ and with Galois group~\disp{\lsD\Gal{\Rd{}}{K}\lsD^{-1}} of order~\disp{[\Rd{} : K]}.
So~\disp{\lsD\left(\Rd{}\right)=\Rd{}} and~\disp{\Rd{}/\QQ} is Galois. 
Similarly for~$\Rd{0}$. 
Finally~(iv),~(v) follow from a consideration of the moduli under the map~$\lsD$. 
\end{proof}

Interestingly,~$\Rd{0}$ is abelian over~$\QQ$ in dimensions~$d=4,5,7,8$. 
However this is never true of~$\Rd{}$. 
Moreover, we always have $\Gal{\Rd{}}{\Rd{0}} \simeq \ZZ/2\ZZ\times\ZZ/2\ZZ$, as we now show. 

\begin{proposition}\label{nonG}
The groups $\Gal{\Rd{}}{\Rd{1}}$, $\Gal{\Rd{1}}{\Rd{0}}$, $\Gal{\Rd{}}{\Rd{2}}$, $\Gal{\Rd{2}}{\Rd{0}}$ all have order~2. 
\end{proposition}
\noindent We shall prove this after pointing out some consequences for the structure of the \SIC\ fields. 
Write~$\g_j$ for the unique non-trivial element of~\disp{\Gal{\Rd{j}}{\Rd{0}}} for~$j=1,2$: that is,~\disp{\g_j\in\Gal{\Rd{j}}{K}} is complex conjugation under every complex embedding of~\disp{\Rd{j}}. 
Let~\disp{\Gamma_{12}} be the subgroup of~$\Gal{\Rd{}}{K}$ of order~2 generated by the product~$\g_1\g_2$, and recall the notation~$M^G$ for the elements of a~$G$-module~$M$ fixed by the action of~$G$. 
By Proposition~\ref{rcfbasics}, any lifting~$\lsD$ of~$\sD$ swaps~$\g_1$ and~$\g_2$. 

\begin{corollary} \label{cor:rayclassprops}
\emph{(i)} $\Rd{}$ is non-abelian over~$\QQ$.

\emph{(ii)} \disp{\Rd{1},\Rd{2}} are non-Galois over~$\QQ$, having~\disp{[\Rd{1}:K]} real places and~\disp{\frac{1}{2}[\Rd{1}:K]} pairs of complex places.  

\emph{(iii)} \disp{\Rd{}=\Rd{1}(\z{k})=\Rd{2}(\z{k})} for any~$k>2$ which is a divisor of~$d'$. 

\emph{(iv)} With~$k$ as in~\emph{(iii)}, \disp{\Rd{0}(\z{k})=\Rd{}^{\Gamma_{12}}} is a CM-field. \disp{\Rd{1},\Rd{2}} and~\disp{\Rd{0}(\z{k})} are the three quadratic fields between~\disp{\Rd{0}} and~\disp{\Rd{}}. 
\end{corollary}

\begin{figure}
\begin{center}
\begin{tikzcd}
& \Rd{} =  \Rd{1} \Rd{2}
\\ 
 \Rd{1}  \arrow[ur, dash,very thick, black,"\ipl{2}"]	&  \Rd{0}(\zeta_{d'}) \arrow[u, dash,very thick, black] 	&	   \Rd{2}  \arrow[ul, dash,very thick, black, swap, "\ipl{1}"] 	
\\
&	\Rd{0} \arrow[ul,dash,very thick,black,"\ipl{1}"] \arrow[u,dash,very thick,black,""]  \arrow[ur,dash,very thick,black,"\ipl{2}"'] 	
\\
\\
&	H_K \arrow[uu,dash,very thick,black,"d'"] 	
\\
 & K  
\arrow[u, dash, very thick, black,"1"] 
= \QQ(\sqrt{D})	
\\
\QQ \arrow[ur, dash,very thick,black] 		
\\
\end{tikzcd}
\end{center}
\vspace{-10mm}
\vbox{
\tiny{
\blue{\caption{\small{Tower of ray class field extensions above~$K$, for a fixed finite modulus~$d'$ and the four possible combinations of infinite places.  
The moduli are `cumulative' going up the tower from~$K$, so for example~$\Rd{1}$ has modulus~$d'\ipl1$. 
The Hilbert class field~$H_K$ of~$K$ is always the ray class field of modulus~$1$.  
The CM field~$\Rd{0}(\zeta_{d'})$ is not strictly a ray class field itself, in general.  
}}}}
\label{Figfieldtower}}
\end{figure}

\begin{proof}(of Corollary). 
If~\disp{\Rd{}/\QQ} were abelian then the inner automorphism~$\lsD$ would be trivial. 
But~\disp{\Rd{}} is the compositum of~\disp{\Rd{1}} and~\disp{\Rd{2}={\Rd{1}}^{\lsD}}, so~\disp{\Rd{}=\Rd{1}=\Rd{2}}, contradicting Proposition~\ref{nonG}. 
This proves~(i).
For~(ii):~\disp{\Rd{1}} is Galois over~$K$ and~\disp{\Rd{1}/\Rd{0}} is non-trivial; so it has~\disp{\frac{1}{2}[\Rd{1}:K]} pairs of complex places over~$\ipl{1}$.  
The corresponding~\disp{[\Rd{1}:K]} places over~$\ipl{2}$ are real. 
This argument is symmetric in~\disp{\Rd{1},\Rd{2}} by Proposition~\ref{rcfbasics}. 
For~(iii) just combine the facts that~\disp{\Rd{1}} is not totally complex,~\disp{\Rd{}/\Rd{1}} has degree~2, and~\blue{\disp{K(\zeta_{d'})\subseteq\Rd{}}}. 
Since~\disp{\z{d'}+\z{d'}^{-1}\in\Rd{0}} this also implies~(iv) . 
\end{proof}

\begin{proof}(of Proposition~\ref{nonG}). 
Let~$h_K$ denote the class number of~$\zk$. 
If~$\mm$ is any modulus of~$K$ we denote by~$U_{\mm}^1$ the subgroup of~$U_K = \zk^\times$ consisting of units~$\modx{1}{\mm}$. 
This has finite index \hbox{$[U_K : U_{\mm}^1]$} in~$U_K$. 
Express~$\mm = \mm_0 \mm_\infty$ as a product of its finite and infinite parts respectively.
We denote by~$\Phi$ the generalized Euler totient function, so~$\Phi(\mm_0)$ is the order of the multiplicative group of the ring~$\zk/\mm_0$. 
By analogy we write~$\Phi(\mm_\infty)$ for the size of the signature group~$\{\pm1\}^{r_{\mm}}$ where~$r_{\mm}$ is the number of real places included in~$\mm_\infty$. 
The formula~\cite[VI \S1]{lang} for the order~$h_{\mm}$ of the ray class group of~$K$ of modulus~$\mm$ is: 
\begin{equation}\label{ordrcg}
h_{\mm} = \frac{h_K \Phi(\mm_0) \Phi(\mm_\infty)}{[U_K : U_{\mm}^1]}\ .
\end{equation}
Let~$\uf$ and~$u_D$ be defined as in section~\ref{pell}. 
Since~$K$ is a real quadratic field,
\[
U_{\mdp}^1 \subseteq U_{d'\ipl{1}}^1 \subseteq U_{d'}^1 \subseteq \zk^\times = \uf^{\ZZ}\times\{\pm1\} 
\]
with each inclusion being of finite index; and similarly with~$U_{d'\ipl{2}}^1$ replacing~$U_{d'\ipl{1}}^1$. 
So for our purposes it suffices to show that~$U_{\mdp}^1 = U_{d'\ipl{1}}^1 = U_{d'\ipl{2}}^1 = U_{d'}^1,$~or equivalently: 
\begin{equation}\label{unitreq}
U_{\mdp}^1 = U_{d'}^1,
\end{equation}
for then the signature factor~$\Phi(\mm_\infty)$ in~(\ref{ordrcg}) will tell the whole story of the growth in the size of the ray class groups as we successively add real places into the modulus. 

Now~(\ref{unitreq}) says that every unit in~$U_{d'}^1$ is totally positive.  
Since~$d'\in\ZZ$,~$U_{d'}^1$ is closed under the action of~$\sD$. 
Hence~$\normN{K}{\QQ}{U_{d'}^1}\subseteq U_{d'}^1$. 
But~$-1\notin U_{d'}^1$ because~$d'>2$. 
So~$U_{d'}^1$ is a rank~1 torsion-free abelian group, all of whose elements have absolute norm~$+1$. 
In particular their signature must either be~$(+,+)$ or~$(-,-)$, so we are reduced to showing that~\emph{no unit in~$U_{d'}^1$ can be totally negative}.  

By construction~$u_D$ is the first totally positive power of~$\uf$, so any totally negative unit in~$U_{d'}^1$ must be of the form~$-u_D^r$ for some~$r\in\ZZ$. 
Suppose that such a unit exists and that~$r\geq1$ is minimal. 
Let $j$ be the positive integer such that $d=d_j$, in the sense of~(\ref{dr}).  
Since both~$u_D^r$ and~$u_D^{-r}\equiv-1\bmod d'_j$ it follows that
\begin{equation}\label{drdj}
d_r = u^r_D + u^{-r}_D +1 \equiv -1 \bmod d'_j.
\end{equation}

There are two cases. 
If~$3 \nmid  j$, then the congruences~\textbf{C1},~\textbf{C2} imply~$d_1 \mid d'_j$ and consequently \hbox{$u_D^r\equiv-1\bmod d_1$}.  
Moreover~(\ref{drdj}) shows~$d_r \equiv -1 \bmod d_1$, so a second application of~\textbf{C0}, \textbf{C1}, \textbf{C2} forces~$d_1=4$ and~$3 \mid r$.  
We therefore have~$D=5$ and~$u_D=\uf^2=\frac{3+\sqrt{5}}{2}$, which has order~3 inside~$\zk/(4)$.  Since $3|r$ this means $u_D^r \equiv 1 \bmod d_1$, a contradiction. 

On the other hand, if~$3 \mid j$ then write~$j=3^t q$ for~$t,q\geq1$ and~$3\nmid q$.  
It follows from~\textbf{C1}, \textbf{C2} that $ d_{3^t} \mid d'_j $, implying $u^r_D\equiv-1 \bmod d_{3^t}$.   
So if such a totally negative unit exists for~$j=3^t q$ then it must also exist for~$j=3^t$.  
Redefine~$r$ to be minimal for~\emph{this} relation: then the order of~$u_D$ modulo~$d_{3^t}$ is~$2r$, which is even. 
But~$d_{3c} = T^{*}_3(d_c) = d^2_c(d_c-3)+3$ for any integer~$c$ and so~$d_{3^t}$ is odd, which by Lemma~\ref{uDorder} below implies that the order of~$u_D$ modulo~$d_{3^t}$ is also odd, a contradiction. 
\end{proof}

\noindent Finally we prove the technical lemma used in the last proof. 

\begin{lemma}\label{uDorder}
The order of~$u_D$ modulo~$d_r'$ is~$3r\frac{d_r'}{d_r}$ (i.e.\ it is~$3r$ if~$d_r$ is odd and~$6r$ if~$d_r$ is even). 
\end{lemma}

\begin{proof}
Consider the absolute minimal polynomial for~$u_D^r$: 
\begin{equation}\label{minp}
X^2-(d_r-1)X+1=(X-u_D^r)(X-u_D^{-r}).
\end{equation} 
Multiplying by~$(X-1)$ shows that~$u_D^{3r}\equiv1\bmod d_r$ for all~$r$. 
When~$d_r$ is even, we know from~\textbf{C2} that~$d_{2r}$ is divisible by~$d_r'=2d_r$ and thus: 
\begin{equation*}
(X-u_D^{2r})(X-u_D^{-2r}) = X^2-(d_{2r}-1)X+1\equiv X^2+X+1 \bmod 2d_r.
\end{equation*}
Once again, multiplying by~$(X-1)$ yields~$u_D^{6r}\equiv1\bmod d_r'$, for any~$r$. 

It remains to show that the power~$3r\frac{d_r'}{d_r}$ is minimal in each case. 
We first show that~$3r$ is minimal modulo~$d_r$ for all~$r$. 
Let~$q\in\NN$ be minimal such that~$u_D^q\equiv1\bmod d_r$. 
It is easy to see that~$q\mid3r$. 
Now the minimal polynomial of~$u_D^q$ is~$X^2-(d_q-1)X+1$, so this must vanish modulo~$d_r$ at~$X=1$. 
In other words~$d_r$ divides into~$d_q-3$, proving that~$r<q$ since~$(d_j)_{j\geq1}$ is a strictly increasing sequence. 
So~$3\mid q$, since otherwise~$q\mid 3r\implies q\mid r\implies q\leq r$. 
Writing~$q=3q_0$ we see that~$q_0\mid r$ and so~$q_0\leq r<3q_0=q$. 
This forces~$q_0=r$ or~$q_0=r/2$ and we now show in fact~$q_0=r$. 
If~$r$ is odd we are done. 
If~$r$ is even it follows from equation~(\ref{minp}) that~\disp{u_D^{-\frac{r}{2}}(u_D^{2r} + u_D^r +1) \equiv 0 \bmod d_r}, which translates via equation~(\ref{dr}) to~\disp{u_D^\frac{3r}{2}\equiv 1-d_\frac{r}{2}\bmod d_r}. 
But~$d_\frac{r}{2}\not\equiv0\bmod d_r$, so again~$q_0=r$. 

Finally assume that~$d_r$ is even and let~$q'\in\NN$ be the minimal integer such that~$u_D^{q'}\equiv1\bmod d_r'=2d_r$. 
Again we must have~$q'\mid 6r$; moreover~$q'<3r$ would (by reduction modulo~$d_r$) contradict the proof above for~$d_r$. 
So either~$q'=3r$ or~$q'=6r$. 
Using~(\ref{dr}),~(\ref{minp}) as above~$u_D^{3r}\equiv1-d_r\bmod d_{2r}$, $d_r\not\equiv0\bmod 2d_r$ and the proof is complete again observing that~\hbox{$d_r$ even $\implies 2d_r\mid d_{2r}$}.
\end{proof}

\section{Linking unitary geometry and arithmetic via subgroups of \texorpdfstring{$\glzd$}{GL2}}\label{GandA}
Let~$\bfv$ be a fiducial vector taken from one of the orbits in the list (\ref{rcfdimz}) that is \emph{centred} as in \S\ref{sec:Weil}.
This section and the next are devoted to a study of the function $f_v \colon (\ZZ/d'\ZZ)^2 \to \Rd{}$ defined by
\begin{equation} \label{eq:fvjdf}
f_v(j) = \frac{\left( \bfv, \Delta_\bfj\bfv\right)}{(\bfv,\bfv)} = \Tr \Pi_{\bfv} \Delta_\bfj,
\end{equation}
where the~$\Delta_\bfj$ are defined as in \eqref{DisOpDef}.  
Unlike the components of the fiducial vector the values of the function $f_v$ are independent of both the scaling and the basis.  
It follows from Proposition~\ref{rcfobs} and (\ref{DisOpDef})  that the values $f_v(j)$ for $\bfj \nequiv \binom{0}{0} \bmod d$ are all elements of~$\Rd{}$ with modulus~$\invnf$.  In fact we have found empirically that it is possible to choose $\bfv$ so that the field~$\QQ(f_v)$ obtained by adjoining to~$\QQ$ all of the values of $f_v$ is equal to~$\Rd{1}$.  We say that a centred fiducial vector for which that is true is \emph{strongly centred}.  Such fiducials may be characterized as follows:
\begin{itemize}
\item If $3 \nmid d$  every centred fiducial is  strongly  centred.
\item If $3\mid d$ let $n = d/3$ and $B = \Delta_{n,2n}$.  Then for each centred fiducial $\bfv$ exactly one of the three vectors $\bfv, B\bfv, B^2 \bfv$ is  strongly centred.
\end{itemize} 
From now on it will be assumed without comment that fiducials are strongly centred.
\blue{
\begin{lemma} \label{lm:fvprm}
If $v$ is a strongly-centred fiducial and $g$ is any element of $\Gal{\mathfrak{R}}{K}$ then $g$ permutes the set $\{ f_v(j)\colon j \in (\ZZ/d'\ZZ)^2\}$.
\end{lemma}
\begin{proof}
Immediate consequence of Theorem~8 in ref.~\cite{AYZ}.
\end{proof}
}
Let~$L$ be the smallest subfield of~$\Rd{1}$ fixing $f_v((\ZZ/d'\ZZ)^2)$ setwise. 
In all of our examples, $L \supset K$ with index $[L:K]$ equal to the number of $\Sp_2(\ZZ/d'\ZZ)$-orbits of fiducials~\cite[\S7]{AYZ}; see the remark at the end of this section. 
In addition to the Galois action, $\glzd$ acts naturally on functions $f \colon (\ZZ/d'\ZZ)^2 \to \Rd{1}$ by translation via $(gf)(j) = f(gj)$.  
What follows is a tantalizing observation linking these two distinct actions,  strengthening a statement in \S7 of~\cite{AYZ}. 
Define
\[
\stabo = \{ (\det F) F \colon F \in \stabo_0 \},
\]
where $\stabo_0 \subset \ESp_2(\ZZ/d'\ZZ)$ is the stabilizer defined in Sec.~\ref{sec:Weil}. 
As shown in \cite{AYZ}, $\boldsymbol{S}$ is the stabilizer of $f_v$, i.e.\ $\stabo = \{g \in \espzd : g f_v = f_v\}$.

\begin{proposition}\label{CS}
For all orbits listed in~\emph{(\ref{rcfdimz})} $L$ is the Hilbert class field~$H_K$ of $K$, and
\begin{equation}\label{arith}
\MS/\,\stabo \ \ \simeq \ \Gal{\Rd{1}}{L},
\end{equation}
where~$\MS$ is a maximal abelian subgroup of~$\glzd$ containing~$\stabo$. For the orbits of $F_z$-type in~\emph{(\ref{rcfdimz})} $\MS$ is in fact $\CS$, the centralizer of~$\stabo$ inside~$\glzd$; for the orbits of $F_a$-type~$\MS$ is conjugate to one of three possible subgroups that are characterized in Ref.~\cite{MAIB}. 
\end{proposition}

We are grateful to John Coates for the following observation, which we hope to address in a forthcoming paper: 

\begin{remark}
The appearance of~$H_K$ in~(\ref{arith}) cannot be a coincidence. 
The LHS is ostensibly a geometrically defined abelian subgroup of $\glzd$. 
On the other hand the RHS contains information about a subgroup of a ray class group, and therefore also potentially about the ideal class group~$\cK$ of~$K$. 
But this cannot be true in general, since the structure of~$\cK$ is very erratic as~$D$ varies. 
We can therefore be confident that $L$ will typically contain~$H_K$. 
\end{remark}
Our empirical observations suggest that $L$ may be identical with~$H_K$ in all dimensions, not just those in (\ref{dimz}). 
It follows that the~$\Gal{L}{K}$-set of distinct $\Sp_2(\ZZ/d'\ZZ)$-orbits may actually be a~$\cK$-set. 
Preliminary numerical results communicated to us by Andrew Scott for certain higher dimensions, wherein~$\cK$ is much larger than for the dimensions in~(\ref{dimz}), provide additional evidence for this speculation.

\section{Canonical units associated to the ray class fields}\label{canon}
We now proceed to link invariants of the \SICs\ with canonical units associated to the ray class fields. 
The condition of equiangularity means that 
\blue{when  \hbox{$\bfj \nequiv \binom{0}{0} \bmod d$}, the inner products  $f_v(j)$ defined in Eq.~\eqref{eq:fvjdf} all have absolute value $\invnf$.  It follows that the numbers
\begin{align}\label{eq:nminnerprd}
\mathrm{e}^{i \theta_v(\bfj)} = \begin{cases} 1 \qquad & \text{if $\bfj \equiv \binom{0}{0} \bmod d$} \\ \sqrt{d+1} f_v(j) \qquad &\text{otherwise} \end{cases}
\end{align} 
are all in $U(1)$.  We refer to them as \emph{normalized inner products}, and they are the subject of this section.  
The fact that we are now restricting ourselves to strongly centred fiducials means that the numbers  $f_v(j) $ are all in $\Rd{1}$.  However, this may not be true of the numbers $\mathrm{e}^{i \theta_v(\bfj)}$.  
We therefore need to introduce the extension field $\Sd=\Rd{1}(\nf)$ and its ring of integers $\zSd$.  
It is also convenient to define $\Sdr = \Rd{}(\nf)$ (see below for the relation between  $\Sd$, $\Sdr$ and $\Rd{1}$, $\Rd{}$).  
We then have }

\blue{
\begin{proposition}\label{units}
Fix any~$d$ in Eq.~\emph{(\ref{dimz})} and consider \blue{a strongly-centred \SIC\ fiducial} on an~$ \EAut^0(H)$-orbit.  
\begin{itemize}\item[(i)] Every one of the normalized inner products lies in the unit group~$\zSd^\times$. \end{itemize}
Let~$u_d$ be a normalized inner product of \emph{maximal} degree~$n_d$ over $K(\nf)$. 
\begin{itemize}
	\item[(ii)] $u_d$ generates\footnote{Except in orbit~12b where we also need to adjoin~$\sqrt{3}$.}\ \ $\Sd$ over~$\QQ(\nf)$ and therefore over~$K(\nf)$.
	\item[(iii)] Given any complex embedding of~$\Sd$, the~$\gsk$-conjugates of~$u_d$ all lie on~$U(1)$. 
	\item[(iv)] The minimal polynomial~$f_d(x)$ of~$u_d$ is a~\emph{reciprocal} polynomial: that is,~$x^{n_d} f_d(\frac{1}{x}) = f_d(x)$. 
\end{itemize}
\end{proposition}
\begin{remark}
For more on the Galois theory of fields generated by reciprocal polynomials see~\cite{lalande}. 
\end{remark}
\begin{proof}
As with the main proposition, this is proven through lengthy computations using \emph{Magma}~\cite{magma}. 
The main steps are again indicated in Section~\ref{very}.
\end{proof}

\begin{corollary}
Let~$u_d$ be a normalized inner product of \emph{maximal} degree~$n_d$ over $K(\nf)$. 
Then all of its images under the action of $\Gal{\Sd}{K(\nf}$ appear as normalized inner products for the same \SIC. 
In particular,~$[\Sd:K(\nf)]=n_d$.
\end{corollary}
\begin{proof}
Immediate consequence of Proposition~\ref{units} and Lemma~\ref{lm:fvprm}.
\end{proof}
We now examine the relation between   $\Sd$, $\Sdr$ and $\Rd{1}$, $\Rd{}$. }
Observe at the outset that~$\nf\in K$ if and only if either~$d+1$ or~$d-3$ is a perfect square. 

\begin{proposition}\label{nfinout}
Suppose $d\geq4$ and neither~$d+1$ nor~$d-3$ is a perfect square. Then 
\begin{itemize}
	\item[(i)] If $d$ is even then~$\nf\in\Rd{0}$.
	\item[(ii)] If $d\equiv1\ (4)$ then~$\nf \not \in \Rd{0}$.
	\item[(iii)] If $d\equiv3\ (4)$ then~$\nf \in \Rd{0}$ iff the extension $K(\nf)/K$ is everywhere unramified.  In particular~$\nf \not\in \Rd{0}$ if the class number of~$K$ is~1 (that is, if~$H_K=K$). 
\end{itemize}
\end{proposition}

\begin{proof}
By assumption~\disp{K(\nf)} is a totally real biquadratic extension of~$\QQ$, with exactly three distinct proper quadratic subextensions~\disp{K=\qr{\qD},\QQ(\nf)} and~\disp{\QQ(\nft)}.  
If~$p$ is an odd prime divisor of the square-free part of~$d+1$ then~$p$ will ramify in~\disp{K/\QQ} but not in~\disp{K(\nf)/K}.  
So~\disp{K(\nf)/K} is unramified at every place not lying over the prime~2.  

When~$d$ is even,~\disp{d'=2d} is divisible by at least~$2^2$ whereas~\disp{d+1} is odd, so the~2-primary part of the conductor of~\disp{K(\nf)/K} divides that of~$d'$~\cite[\S17F]{harv}.  
So~$\nf\in\Rd{0}$, proving~(i).  

When~$d$ is odd, so~$d'=d$, the only cases where~$\nf$ can possibly lie in the ray class field~$\Rd{0}$ are those where~2 does not ramify in the extension~\disp{K(\nf)/K}.  
This proves assertion~(iii) for~\emph{any} odd~$d$, with the consequence that the extension~~$H_K/K$ must be non-trivial if~$\nf \in \Rd{0}$.  

To prove~(ii), if~$d\equiv1\ \bmod 4$ then we claim that the unique prime of~$K$ above~2 always ramifies again in~\disp{K(\nf)/K}: meaning that~$\nf$ can never lie in~$\Rd{0}$.  
First of all,~\disp{\QQ(\nf)/\QQ} and \disp{\QQ(\nft)/\QQ} are each ramified over~2, since both~\disp{d+1} and~\disp{d-3\equiv2\ \bmod 4}.   
Also~\disp{(d+1)(d-3)\equiv12\  \bmod 16} and since~\disp{2^2\mid(d+1)(d-3)} 
it follows that~\disp{D\equiv3\ \bmod 4}.  
So~2~is ramified in all three quadratic subextensions~\disp{\qr{\qD},\ \QQ(\nf),\ \QQ(\nft)}. 
Therefore~\cite[\S17F]{harv} the inertia group at~2 must be the whole Galois group~$\Gal{K({\tinynf})}{\QQ}$, proving the claim.   
\end{proof}

\begin{corollary}\label{f1eqr1dims}
For the dimensions in list~(\ref{dimz}) it follows that~$\nf\in\Rd{0}$ except for~$d=5$, $9$, $11$, $13$, $17$, $21$. 
\emph{(}The case $d=11$ follows from the fact that the class number of~$\QQ(\sqrt{6})$ is~1\emph{)}. 
\end{corollary}
\blue{\noindent Since~$\Rd{0}$ is the maximal totally real subfield of~$\Rd{}$, it follows therefore that~$\nf \notin \Rd{}$ for  $d=5$, $ 9$, $11$, $13$, $17$, $21$.  We conclude
\begin{itemize}
\item If $d=4$, $6$--$8$, $10$, $12$, $14$--$16$, $18$--$20$, $24$ , $28$, $30$, $35$, $39$, $48$ then $\Sd = \Rd{1}$, $\Sdr =\Rd{}$.
\item If $d=5$, $9$, $11$, $13$, $17$, $21$ then $\Sd$, $\Sdr$ are degree 2 extensions of $\Rd{1}$, $\Rd{}$ such that $\Sd \cap \Rd{}=\Rd{1}$. 
\end{itemize}
}
The next result is striking but the size of the necessary calculations has prevented us from testing dimensions~$16,17,18,20,21,30$ and~$39$ in list~(\ref{dimz}). Let~\disp{\UU} denote the group of units of the ring of integers of the field~\disp{\Sd} under any one of its complex embeddings, and let~\disp{\VV} be the subgroup generated by the normalized inner products. 
By construction~\disp{\VV} is contained in the unit circle subgroup~\disp{\UU\cap U(1)} of~\disp{\UU}. 
We would like to compare the relative ranks of~\disp{\VV} and~\disp{\UU\cap U(1)}. 
Let~\disp{N_d=[\, \Sd:K\,]}, so~\disp{N_d=n_d} or~\disp{2n_d} according to whether~\disp{\nf\in\Rd{0}} or not. 
Part~(ii) of the Corollary to~Proposition~\ref{nonG} implies that~\disp{\Sd} has~\disp{N_d} real places and~\disp{\tfrac{N_d}{2}} pairs of complex places; so by Dirichlet's unit theorem~\disp{\UU} is an abelian group of rank~\disp{\tfrac{3N_d}{2}-1}. 
By the same Proposition~\disp{\Rd{0}(\nf)} is the maximal real subfield of~\disp{\Sd} and has index~2: so it is a consequence of the lemma in~\S5 of~\cite{naka} that the unit circle subgroup~\disp{\UU\cap U(1)} of~\disp{\UU} has~\disp{\ZZ}-rank~\disp{\tfrac{N_d}{2}}.

\blue{
The following proposition is split into two parts.  The first part holds for the set of   dimensions $4$--$15$, $19$, $24$, $28$, $35$ and $48$ for which  we were able to calculate the group  $\VV$.  The second part holds for the smaller set of dimensions $4$--$8$, $12$, $19$ for  which we were able to calculate both the group $\VV$ and the group $\UU$.
\begin{proposition}\label{prp:ugp}
(A) In dimensions $4$--$15$, $19$, $24$, $28$, $35$ and $48$,  $\VV$ is generated by the $\gsk$-orbit of~$u_d$.  One has
\begin{align}
\rank(\VV) = 
\begin{cases}
 \rank (\UU\cap U(1)), & d = 7,15,19,35; \\
 \frac{1}{2} \rank (\UU\cap U(1)), & d = 4,5,6,8,9,10,11,12,13,14,24,28,48. 
\end{cases}
\end{align}

(B) In dimensions $4$--$8$, $12$ and $19$, $\VV$ is also a direct summand of ~\hbox{$\UU\cap U(1)$}. 
In particular,~\hbox{$\UU\cap U(1)=\VV$} in dimensions $7$ and $19$.
\end{proposition}
\begin{proof}
Again, this is proven via lengthy computations using \emph{Magma}~\cite{magma}, and the main steps are indicated in Section~\ref{very}. 
Note that the calculation of $\UU$, in those dimensions where it was possible, used the function \emph{IndependentUnits} in \emph{Magma}, which assumes the validity of the generalized Riemann hypothesis.
\end{proof}
}

Finally, preliminary calculations by one of us (GM) together with Steve Donnelly and reported in the first draft of this paper indicated that  the logarithms of these canonical units may be related to the values of~$L$-functions associated to the extensions, following the programme laid out in the Stark Conjectures.  \blue{Subsequent work by Gene Kopp has numerically confirmed this conjecture for dimension~$d=5$~\cite{kopp} and some higher dimensions~\cite{kopp2}, via the values at~$s=0$ of the first derivatives of the partial L-functions attached to the ray classes. 
}

\section{Verification of Propositions~\ref{rcfobs},~\ref{prpconstfield},~\ref{units}, and~\ref{prp:ugp}}\label{very}
\blue{A full account of the verification of these propositions would be very tedious. In this section we confine ourselves to an indication of the main steps.

\subsection{Propositions~\ref{rcfobs},~\ref{prpconstfield}}
The first obstacle we face is the definition of the field $\QQ(E)$, in terms of the ratios of components of the vectors $v$ defining the lines.  
Conceptually speaking this is the most natural way of associating a field to a set of equiangular lines.  
However, it is cumbersome to work with.  
To obtain a computationally more convenient definition we work with the projectors
\begin{align}
\Pi_v = \frac{vv^\dagger}{(v,v)}
\end{align}
instead of the vectors $v$. Let $\QQ(\Pi_v)$ be the field generated by $\frac{v_j v^{*}_k}{(\bfv, \bfv)}$ of $\Pi_v$, and let $L$ be the smallest subfield of $\mathbb{C}$ containing  $\QQ(\Pi_v)$ for every $\mathbb{C}v\in E$.  It is straightforward to verify that $L$ is Galois if and only if $\QQ(E)$ is Galois, in which case $L = \QQ(E)$. 
We can further simplify the computations by observing the following. 
The fact that the lines form an orbit under the action of the group generated by the matrices $X_d$, $Z_d$ specified by Eq.~\eqref{eq:nminnerprd} means that $L  = \QQ(\Pi_v, \zeta_d)$, the field generated by the matrix elements of the single projector $\Pi_v$ together with the $d^{\rm{th}}$ root of unity $\zeta_d = e^{2\pi i /d}$.  
The problem of verifying Proposition~\ref{rcfobs} thus reduces to showing that for each dimension $d$ in the list~\eqref{dimz} there is a fiducial projector $\Pi_v$ such that  $\QQ(\Pi_v, \zeta_d)$ is the ray class field $\mathfrak{R}$.

At this point we ought to comment on a seeming discrepancy between what is said in the previous paragraph, and what is said in, for example, ref.~\cite{MAIB}, where the SIC field is defined to be $\QQ(\Pi_v, \zeta_{2d})$.  
In fact,  the discrepancy is only apparent, since for every known SIC the two fields are  identical.  
Indeed, it is immediate that   $\QQ(\Pi_v, \zeta_{2d}) = \QQ(\Pi_v, \zeta_{d})$ when $d$ is odd, since in that case one has $\zeta_{2d} = -\zeta_d^{\frac{d+1}{2}}$.   
If, on the other hand, $d$ is even we appeal to the fact that every known  SIC fiducial $\Pi_v$  is stabilized by a canonical order 3 Clifford unitary~\cite{MAIB,bos} to deduce that
\begin{align}
&\QQ\big( \{ \zeta^{k}_d \Tr(\Delta_\bfj ) \Pi_v \colon k, j_1,j_2 \in \mathbb{Z}/(d\mathbb{Z}) \text{ and } j_1 j_2 =0 \mod 2 \}\big)
\\
&\hspace{4 cm}=
\QQ\big( \{ \zeta^{k}_d \Tr(\Delta_\bfj ) \Pi_v \colon k, j_1,j_2 \in \mathbb{Z}/(d\mathbb{Z})\}\big),
\nonumber
\end{align}
and to deduce in turn that
\begin{align}
\zeta_{2d} \Tr(X_d Z_d \Pi_v)  &\in \QQ\big( \{ \zeta^{k}_d \Tr(\Delta_\bfj ) \Pi_v \colon k, j_1,j_2 \in \mathbb{Z}/(d\mathbb{Z}) \text{ and } j_1 j_2 =0 \mod 2 \}\big)
\\
&
\subseteq \QQ(\Pi_v,\zeta_d).
\nonumber
\end{align}
Since $\Tr(X_d Z_d \Pi_v)$ is also in $\QQ(\Pi_v,\zeta_d)$ it follows that $\zeta_{2d} \in \QQ(\Pi_v,\zeta_d)$ and, consequently, that $\QQ(\Pi_v,\zeta_{2d}) = \QQ(\Pi_v,\zeta_d)$.  
This point incidentally makes clear what might otherwise be doubted: that the introduction of $(2d)^{\rm{th}}$ roots of unity is unavoidable when $d$ is even.

It is shown in ref.~~\cite{MAIB} that $\QQ(\Pi_v , \zeta_{2d})$ is invariant under the action of the extended Clifford group.  
The argument in the last paragraph means that the same is true of $\QQ(\Pi_v , \zeta_{d}) $. 
This proves Proposition~\ref{prpconstfield}. 

The remainder of the verification of Proposition~\ref{rcfobs} is  most conveniently explained by means of an example.   The expressions for fiducial vectors are usually extremely complicated, as can be seen from a casual inspection of refs.~\cite{exactSICsDB,scottgrassl,zauner}.  However, in dimension $7$ one has a SIC fiducial vector with the remarkably simple formula\footnote{\blue{There exists a fiducial vector in dimension~19 with a similarly simple formula~\cite{appleby05,scottgrassl}. The role of the Legendre symbol in these expressions is intriguiing, as is the fact that every component apart from the first has the same absolute value.  For the question, whether there exist other examples of such fiducials, see ref.~\cite{khatirinejad}}. }
\begin{align}
v &= \begin{pmatrix} v_0 \\ v_1 \\ \vdots \\ v_{6} \end{pmatrix},
\label{eq:fid7}
\end{align}
where
\begin{align}
v_n &= \begin{cases}
-2\sqrt{2} \qquad & \text{if $n= 0$}, \\ 1+ i  \genfrac{(}{)}{0.5 pt}{0}{n}{7}\sqrt{4\sqrt{2}-5} \qquad & \text{otherwise},
\end{cases}
\end{align}
$\genfrac{(}{)}{0.5 pt}{0}{n}{7}$ being the Legendre symbol.  The formula being so simple, this is the natural example to choose.

Using \emph{Magma}~\cite{magma} we compute the ray class fields $\mathfrak{R}_1$, $\mathfrak{R}$.  We find that $\mathfrak{R}$ is a degree $2$ extension of $K(\zeta_7)$.  In particular $[\mathfrak{R}\colon K] = 12$, $[\mathfrak{R}_1\colon K] = 6$. We then calculate the matrix elements of $\Pi_v$ and show that they all belong to $\mathfrak{R}$.  It follows that $\QQ(\Pi_v, \zeta_7)\subseteq \mathfrak{R}$.  Finally, we show   that $\QQ(\Pi_v, \zeta_7)$ is a degree $2$ extension of $K(\zeta_7)$.  We conclude that $\QQ(\Pi_v, \zeta_7) = \mathfrak{R}$, thereby verifying Proposition~\ref{rcfobs} for the case $d=7$.

\subsection{Propositions~\ref{units},~\ref{prp:ugp}} We begin by observing that~(i) \& (ii)\ $\implies$(iii)$\implies$(iv).  Indeed, the abelian structure imposes a unique complex conjugation operator~$\sigma_c\in\gsk$ upon every complex embedding. 
So for any~$g\in\gsk$,
\[
\sigma_c(g(u_d))=g(\sigma_c(u_d))=g(u_d^{-1})=(g(u_d))^{-1}.
\]
Since~$\gsk$ is transitive on the roots,~$f_d$ is a product of reciprocal polynomials of the form \hbox{$x^2-(u+u^{-1})x+1$}, hence is itself reciprocal. 

We next show that if (i) and (ii) hold for one strongly-centred fiducial  on a given $ \EAut^0(H)$-orbit then they hold for all.   Indeed, let $v$, $v'$ be two strongly centred fiducials on the same orbit.  It follows from the analysis in ref.~\cite{MAIB} that 
\begin{align}
\big\{ e^{i\theta_{v'}(\bfj)} \colon j \in  (\ZZ/d'\ZZ)^2 \big\} = \big\{ \zeta^{k_1 j_2 - k_2 j_1}_d e^{i\theta_{v}(p(\bfj))} \colon j \in  (\ZZ/d'\ZZ)^2 \}
\end{align}
for some $k\in (\ZZ/d'\ZZ)^2 $ and permutation $p$ fixing the point $\binom{0}{0}$.  The fact that both fiducials are strongly centred means $\zeta_d^{k_1}$, $\zeta^{k_2}_d\in \mathfrak{R}_1$.  Let $g$ be any element of $\Gal{\mathfrak{R}}{\QQ}$ mapping $\mathfrak{R}_1$ to $\mathfrak{R}_2$.  Then $g(\zeta_d^{k_j}) = \zeta_d^{\lambda k_j}$ for some integer $\lambda$ coprime to $d$.  In view of Corollary~\ref{cor:rayclassprops} this means that $\zeta_d^{\lambda k_1}$, $\zeta_d^{\lambda k_2}$  are both real.  It follows that if $k_1$, $k_2$ are both zero if $d$ is odd, and multiples of $d/2$ if $d$ is even.  We conclude that
\begin{align}
\big\{ e^{i\theta_{v'}(\bfj)} \colon j \in  (\ZZ/d'\ZZ)^2 \big\} = \big\{ s_j e^{i\theta_{v}(p(\bfj))} \colon j \in  (\ZZ/d'\ZZ)^2 \}
\end{align}
for some set of signs $s_j$.  The claim is now immediate.

The problem thus reduces to establishing parts (i) and (ii) for one particular strongly centred fiducial.  Since the fiducial specified by Eq.~\eqref{eq:fid7} is strongly-centred we illustrate the calculations with that.  It follows from Corollary~\ref{f1eqr1dims} that $\Sd = \Rd{1}$, $\Sdr =\Rd{}$.  Also we know from the previous sub-section that  $[\mathfrak{R}_1\colon K] = 6$, $[\mathfrak{R}\colon K] = 12$.  Using the function \emph{IndependentUnits} in \emph{Magma} we find that $\UU\cap U(1)$ is generated by
\begin{align}
v_1 &= \frac{1}{14}\left( \left(7+i \sqrt{35+28\sqrt{2}}\right) + \left(-14-i\sqrt{-308+224\sqrt{2}}\right) \cos\left(\frac{\pi}{7}\right) \right.
\\
&\hspace{6 cm} \left.
+\left(14\sqrt{2}+i\sqrt{-56+112\sqrt{2}}\right)\cos\left(\frac{2\pi}{7}\right) \right)
\nonumber
\\
v_2&= \frac{1}{7} \left( i \sqrt{-7+14\sqrt{2}} +\left(7-7\sqrt{2}-i\sqrt{-7+14\sqrt{2}}\right)\cos\left(\frac{\pi}{7}\right) \right.
\\
&\hspace{6 cm} \left. + \left(-7-i\sqrt{-77+56\sqrt{2}}\right)\cos\left(\frac{2\pi}{7}\right) \right)
\nonumber
\\
v_3&= \frac{1}{14} \left(\left(7-7\sqrt{2} + i \sqrt{-7+14\sqrt{2}}\right) +\left(14\sqrt{2} +i \sqrt{-56 +112\sqrt{2}}\right)  \cos\left(\frac{\pi}{7}\right)
\right.
\\
&\hspace{6 cm} \left. \left(14-14\sqrt{2} - i \sqrt{-28+56\sqrt{2}}\right)\cos\left(\frac{2\pi}{7}\right) \right)
\nonumber
\end{align}
together with $-1$.  On the other hand the normalized inner products take the 9 distinct values
\begin{align}
1,\quad v^{\vphantom{*}}_1, \quad v^{\vphantom{*}}_2,\quad v^{\vphantom{*}}_3,\quad -v^{\vphantom{*}}_1v^{\vphantom{*}}_2 v^{\vphantom{*}}_3,\quad v_1^{*},\quad v_2^{*},\quad v_3^{*},\quad -v_1^{*}v_2^{*}v_3^{*}
\end{align}
where $v_j^{*}$ is the complex conjugate of $v^{\vphantom{*}}_j$.  This establishes Proposition~\ref{prp:ugp} and part (i) of Proposition~\ref{units} for the case $d=7$.  Using \emph{Magma} one finds that the minimal polynomials over $K$ of these numbers are
\begin{align}
f_1(x) &= (x-1),
\\
f_2(x) & = \left(x+v^{\vphantom{*}}_1v^{\vphantom{*}}_2 v^{\vphantom{*}}_3\right) \left(x+v_1^{*}v_2^{*}v_3^{*}\right)
\\
&= x^2 + \left(\sqrt{2} - 1\right)x+ 1,
\nonumber
\\
f_3(x) &= \left(x-v^{\vphantom{*}}_1\right)\left(x-v^{\vphantom{*}}_2\right)\left(x-v^{\vphantom{*}}_3\right)\left(x-v^{*}_1\right)\left(x-v^{*}_2\right)\left(x-v^{*}_3\right)
\\
&=
x^6 + \left(\sqrt{2} - 2\right)x^5 + \left(\sqrt{2} - 2\right)x^4 + \left(-5\sqrt{2} + 9\right)x^3 
\nonumber
\\
& \hspace{5 cm} + \left(\sqrt{2} - 2\right)x^2 + \left(\sqrt{2} - 2\right)x + 1.
\nonumber
\end{align}
The fact that $f_3(x)$ is irreducible, has the same degree as $\mathfrak{R}_1/K$ and splits over $\mathfrak{R}_1$ means that its roots generate $\mathfrak{R}_1$.  Since $K = K(\sqrt{d+1}) = \QQ(\sqrt{d+1})$ this proves part (ii) of Proposition~\ref{units} for the case $d=7$. 

Finally, let us remark that in those dimensions where it is impracticable to compute $\UU$ it is still possible to compute $\VV$ by applying an integer relation algorithm to the logarithms of the normalized inner products.
}

\section*{Acknowledgments}
It is a pleasure to thank John Coates and Andrew Scott for their continuing support and guidance. 
Also thanks to Steve Donnelly for his help with the \emph{Magma} code and with interpreting the number-theoretic results; to James McKee and Chris Smyth for explaining various aspects of the theory of reciprocal units; and to Brian Conrad and to an anonymous referee for comments on an earlier draft.

This research was supported by the Australian Research Council via EQuS project number CE11001013, and SF acknowledges support from an Australian Research Council Future Fellowship FT130101744 and JY from National Science Foundation Grant No.\ 116143.
This research was supported in part by Perimeter Institute for Theoretical Physics.  Research at Perimeter Institute is supported by the Government of Canada through Innovation, Science and Economic Development Canada and by the Province of Ontario through the Ministry of Research, Innovation and Science.

\end{document}